\documentclass[reqno]{amsart}

\usepackage{amsmath,amssymb,amsthm}

\usepackage{tikz}

\usepackage{caption}
\usepackage{graphicx}
\usepackage{graphics}

\newtheorem{theorem}{Theorem}

\newtheorem*{lemma}{Lemma}
\newtheorem{corollary}{Corollary}
\newtheorem{proposition}{Proposition}

\begin{document}

\title[]{Greedy Matching in Optimal Transport\\ with  Concave Cost}

\author[]{Andrea Ottolini \and Stefan Steinerberger}

\address{Cowles Foundation, Yale University}
 \email{andrea.ottolini@yale.edu }

\address{Department of Mathematics, University of Washington, Seattle}
 \email{steinerb@uw.edu}

\begin{abstract}
    We consider the optimal transport problem between a set of $n$ red points and a set of $n$ blue points  subject to a concave cost function such as $c(x,y) = \|x-y\|^{p}$ for $0< p < 1$. Our focus is on a particularly simple matching algorithm: match the closest red and blue point, remove them both and repeat. We prove that it provides good results in any metric space $(X,d)$ when the cost function is $c(x,y) = d(x,y)^{p}$ with $0 < p < 1/2$. Empirically, the algorithm produces results that are \textit{remarkably} close to optimal -- especially as the cost function gets more concave; this suggests that greedy matching may be a good toy model for Optimal Transport for very concave transport cost.  
\end{abstract}

\subjclass[2020]{82B44, 90B80} 
\keywords{Optimal Transport, Euclidean Random Assignment Problems}

\maketitle

\vspace{0pt}

\section{Introduction}
\subsection{The problem.} The original motivation behind this paper is to understand the geometry of optimal transport with concave cost. Perhaps the easiest instance of this problem is the following: let
$X = \left\{x_1, \dots, x_n\right\}$ and $Y = \left\{y_1, \dots, y_n\right\}$ be two sets of real numbers, what can be said about the optimal transport cost
$$W^p_{p}(X,Y) = \min_{\pi \in S_n} \sum_{i=1}^{n} | x_i - y_{\pi(i)}|^{p},$$
where $\pi:\left\{1,2,\dots, n\right\} \rightarrow \left\{1,2,\dots, n\right\}$ ranges over all permutations? The answer is trivial when $p \geq 1$: order both sets in increasing order and send the $i-$th largest element $x_i$ to the $i-$th largest element $y_i$. After ordering the points, the optimal permutation is the identity permutation $\pi(i) = i$. The problem becomes highly nontrivial when the cost function is concave.
\vspace{-5pt}
\begin{center}
\begin{figure}[h!]
    \begin{tikzpicture}[scale=0.95]
        \node at (0,0) {\includegraphics[width=0.5\textwidth]{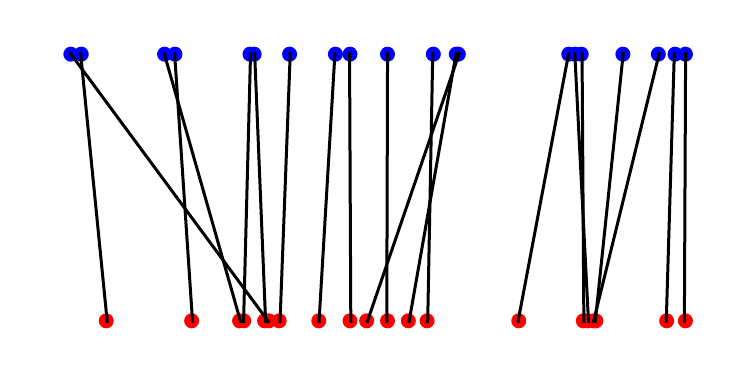}};
        \node at (7,0) {\includegraphics[width=0.5\textwidth]{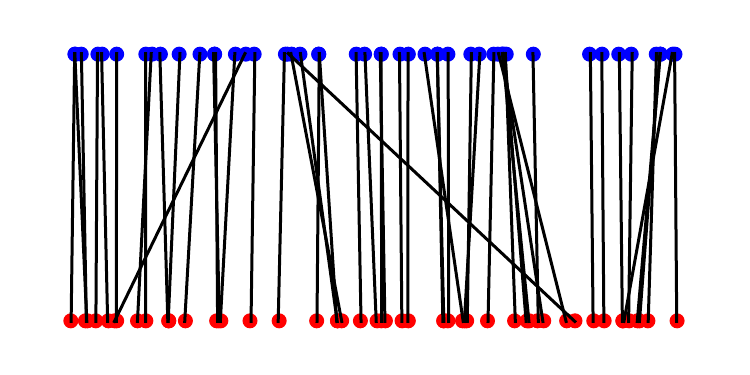}};
    \end{tikzpicture}
    \vspace{-20pt}
    \caption{20 (right: 50) red points on $\mathbb{R}$ being optimally matched to 20 (right: 50) blue points on $\mathbb{R}$ (shown displaced to illustrate the matching) and subject to cost $c(x,y) = |x-y|^{1/2}$.}
\end{figure}
\end{center}

As was already pointed out by Gangbo \& McCann \cite{gangbo} 
\begin{quote}
For concave functions of the distance, the picture which emerges is rather different. Here the optimal maps will not be smooth, but display an intricate structure which --
for us -- was unexpected; it seems equally fascinating from the mathematical and the
economic point of view. 
[...]
To describe one effect in economic terms: the concavity of the cost function
favors a long trip and a short trip over two trips of average length [...] it can be
efficient for two trucks carrying the same commodity to pass each other traveling opposite
directions on the highway: one truck must be a local supplier, the other on a longer haul. (Gangbo \& McCann, \cite{gangbo})
\end{quote}
The problem has received increased attention in recent years, we refer to results of Bobkov and Ledoux \cite{bob, bob2},
Boerma, Tsyvinski, Wang and Zhang \cite{boerma},
Caracciolo, D’Achille, Erba and Sportiello \cite{c},
Caracciolo, Erba and Sportiello \cite{c2, c3}, Delon, Salomon and Sobolevski \cite{delon, delon2}, Juillet \cite{juillet} and  McCann \cite{mccann}.

\vspace{-10pt}
\begin{center}
\begin{figure}[h!]
    \begin{tikzpicture}
        \node at (0,0) {\includegraphics[width=0.45\textwidth]{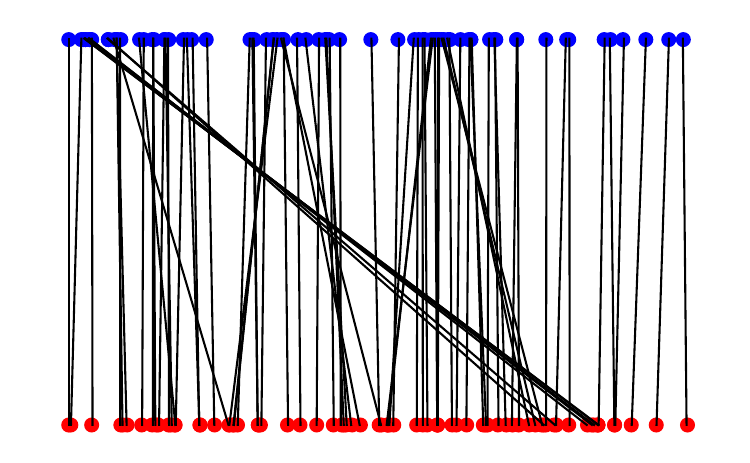}};
        \node at (6.5,0) {\includegraphics[width=0.45\textwidth]{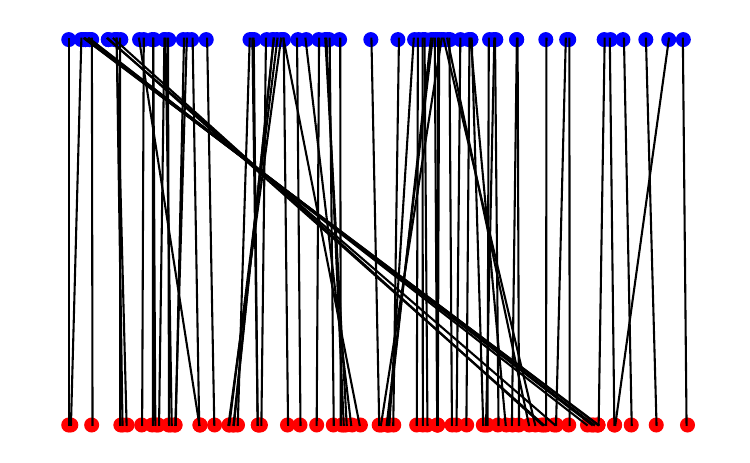}};
    \end{tikzpicture}
    \vspace{-10pt}
    \caption{Left: 50 red points on $\mathbb{R}$ being matched to 50 blue points on $\mathbb{R}$ when $c(x,y) =  \log |x-y|$. Right: same points and same cost function matched with the greedy algorithm. The two matchings are very similar: why?}
    \label{fig:2}
\end{figure}
\end{center}
\vspace{-10pt}
A reason why the problem is interesting is illustrated in Figure \ref{fig:2}: as suggested by Gangbo-McCann, there is a very curious dichotomy where most points get matched to points that are very close with a few exceptional points being transported a great distance. It is somewhat clear, in a qualitative sense, that this is to be expected (considering, for example, the Jensen inequality for concave functions). However, on a more quantitative level, the non-locality poses considerable difficulties.

\subsection{Dyck and Greedy Matching.} If the cost function is given by $c(x,y) = h(|x-y|)$ with $h$ concave, it appears there exist two natural toy models that are effective in different regimes: the Dyck matching and the greedy matching.

\begin{center}
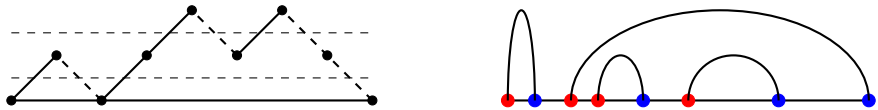
\begin{figure}[h!]
\begin{tikzpicture}[scale=1.2]
\draw [dashed] (0.5, 0.25) -- (4.5, 0.25);
\draw [dashed] (0.5, 0.75) -- (4.5, 0.75);
    \draw [thick] (0+0.5,0) -- (4+0.5,0);
    \filldraw (0+0.5,0) circle (0.05cm);
    \draw[thick] (0+0.5,0) -- (0.5+0.5, 0.5);
        \filldraw (0.5+0.5,0.5) circle (0.05cm);
    \draw[thick, dashed] (1+0.5,0) -- (1, 0.5);
            \filldraw (1+0.5,0) circle (0.05cm);
    \draw[thick] (1+0.5,0) -- (2+0.5, 1);
        \filldraw (1.5+0.5,0.5) circle (0.05cm);
   \filldraw (2+0.5,1) circle (0.05cm);
   \draw[thick, dashed] (2.5+0.5,0.5) -- (2+0.5, 1);
      \filldraw (2.5+0.5,0.5) circle (0.05cm);
   \draw[thick] (2.5+0.5,0.5) -- (3+0.5, 1);
      \filldraw (3+0.5,1) circle (0.05cm);
   \draw[thick, dashed] (4+0.5,0) -- (3+0.5, 1);
         \filldraw (3.5+0.5,0.5) circle (0.05cm);
      \filldraw (4+0.5,0) circle (0.05cm);
          \draw [thick] (6,0) -- (10,0);
      \filldraw[red] (6,0) circle (0.07cm);    
     \filldraw[blue] (6.3,0) circle (0.07cm); 
           \filldraw[red] (6.7,0) circle (0.07cm); 
                 \filldraw[red] (7,0) circle (0.07cm); 
      \filldraw[blue] (7.5,0) circle (0.07cm); 
                       \filldraw[red] (8,0) circle (0.07cm); 
                   \filldraw[blue] (9,0) circle (0.07cm);
    \filldraw[blue] (10,0) circle (0.07cm); 
 \draw[thick] (6,0) arc
    [
        start angle=180,
        end angle=0,
        x radius=0.15cm,
        y radius =1cm
    ] ;
     \draw[thick] (6.7,0) arc
    [
        start angle=180,
        end angle=0,
        x radius=1.65cm,
        y radius =1cm
    ] ;
     \draw[thick] (7,0) arc
    [
        start angle=180,
        end angle=0,
        x radius=0.25cm,
        y radius =0.5cm
    ] ;
     \draw[thick] (8,0) arc
    [
        start angle=180,
        end angle=0,
        x radius=0.5cm,
        y radius =0.5cm
    ] ;
    \end{tikzpicture}
    \caption{A collection of $n=5$ red and blue points, the function $g(x)$ (left, rescaled for clarity) and the Dyck matching (right).}
     \label{fig:dyck}
    \end{figure}
\end{center}

The first such model is the \textit{Dyck matching} of Caracciolo-D’Achille-Erba-Sportiello \cite{c}: their idea is to introduce $g:[0,1] \rightarrow \mathbb{Z}$
$$ g(x) = \# \left\{1 \leq i \leq n: x_i \leq x\right\} - \# \left\{1 \leq i \leq n: y_i \leq x\right\}  $$
The function is increasing whenever $x$ crosses a new element of $X$ while it decreases every time it crosses an element of $Y$. The Dyck matching is then obtained by matching across level sets of the function $g$ (see Fig. \ref{fig:dyck}). The Dyck matching is independent of the cost function.
It is shown numerically in \cite{c} (and reproduced in \S 2.6) that the Dyck matching produces a nearly optimal matching whose costs exceeds the optimal cost by very little. 
The second toy model is given by a simple greedy matching which works in general metric spaces.
\begin{quote}
\textbf{Greedy Matching.} 
\begin{enumerate}
\item Determine
$$ m = \min_{1 \leq i,j \leq n} c(x_i, y_j).$$
\item Find a pair $(x_i, y_j)$ with $c(x_i, y_j) =m$ and set $\pi(i) = j$.
\item Remove $x_i$ from $X$ and $y_j$ from $Y$ and repeat.
\end{enumerate}
\end{quote}

If the cost function is strictly monotonically increasing in the distance $c(x,y) = h(|x-y|)$, this greedy matching is, like the Dyck matching, independent of the cost function. This algorithm leads to mediocre results when the cost function is convex. This was already observed in the PhD thesis of d'Achille \cite[Section 1.4]{achille} who explicitly considers the algorithm when $c(x,y) = |x-y|^p$ and $p\geq 1$ and shows that the results are not particularly good. 
One of the main points of our paper is to point out that the greedy matching is \textit{very} good for very concave cost functions.

\vspace{-5pt}

\begin{center}
    \begin{figure}[h!]
        \begin{tikzpicture}
            \node at (0,0) {\includegraphics[width=0.4\textwidth]{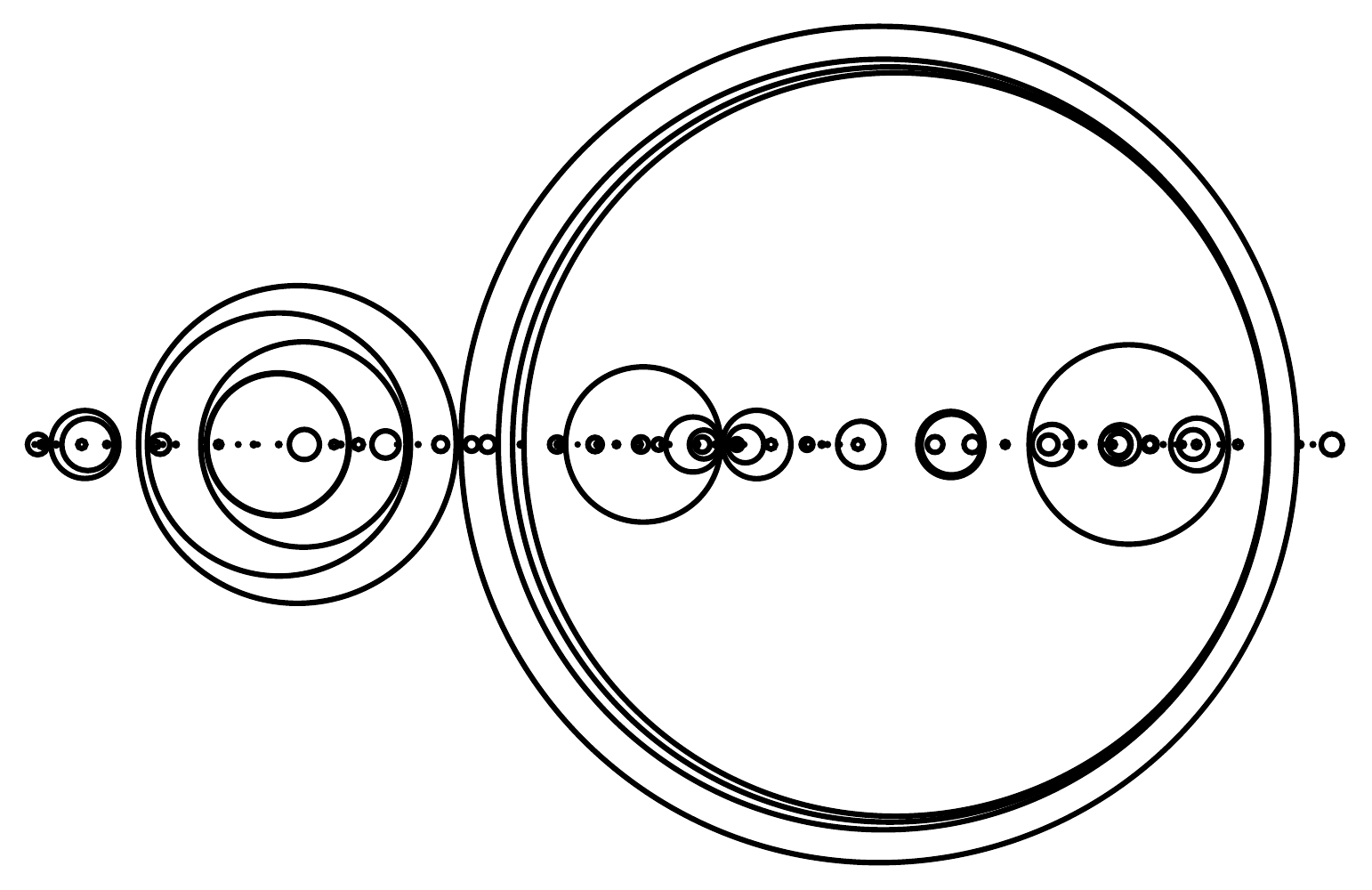}};
            \node at (0.9,0.65) { $|x-y|^{0.01}$};
          \node at (6,0) {\includegraphics[width=0.4\textwidth]{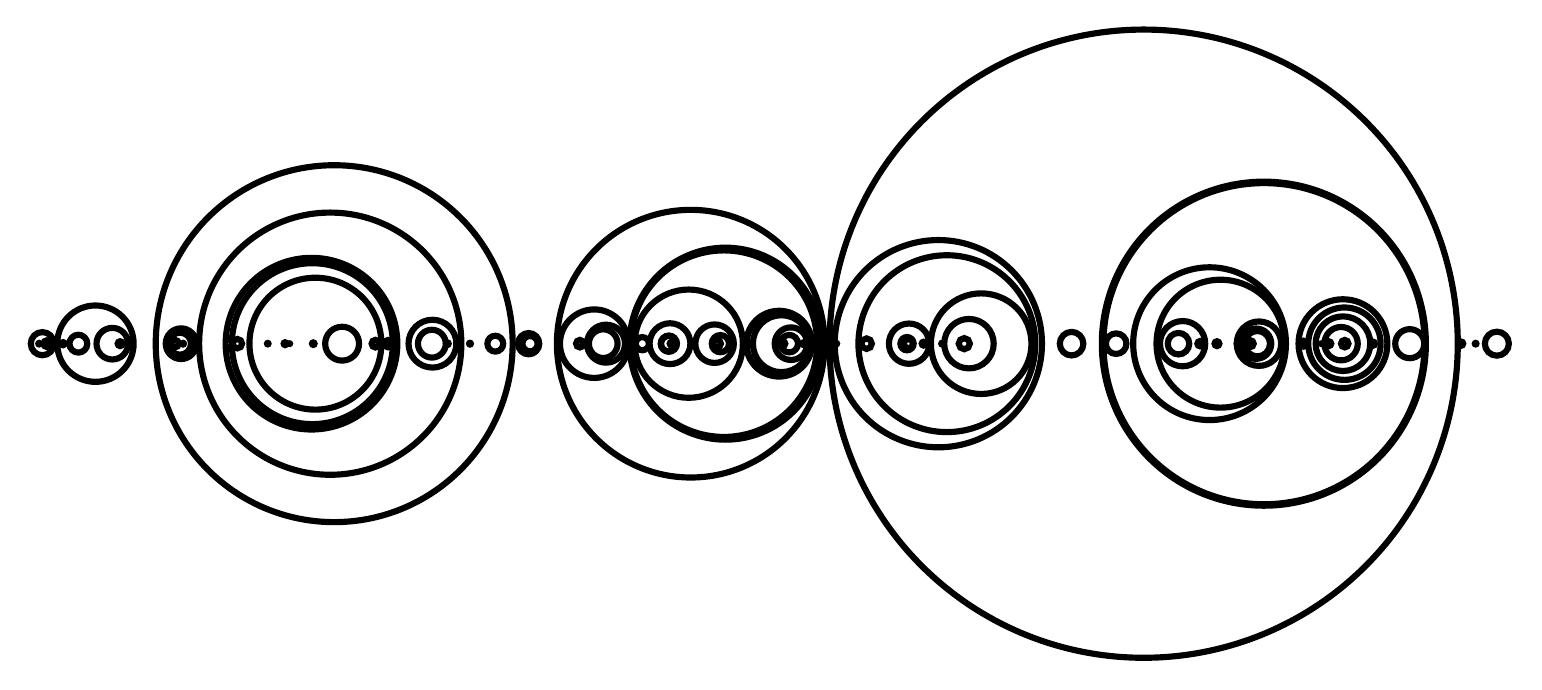}};
            \node at (5.9,1) { $|x-y|^{0.99}$};
     \node at (0,-3.75) {\includegraphics[width=0.4\textwidth]{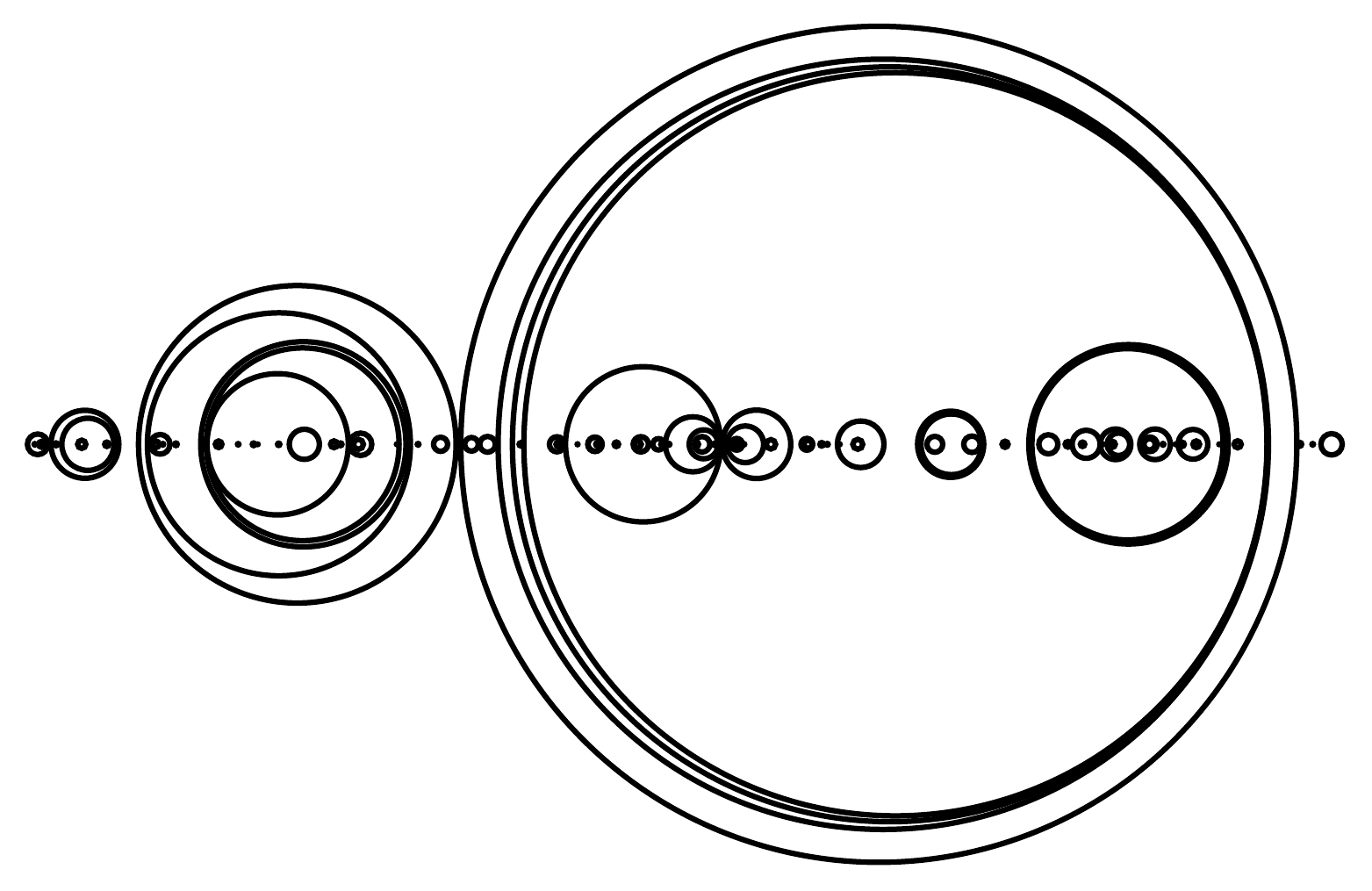}};
         \node at (0.7,-3.1) {greedy};
     \node at (6,-3.75) {\includegraphics[width=0.4\textwidth]{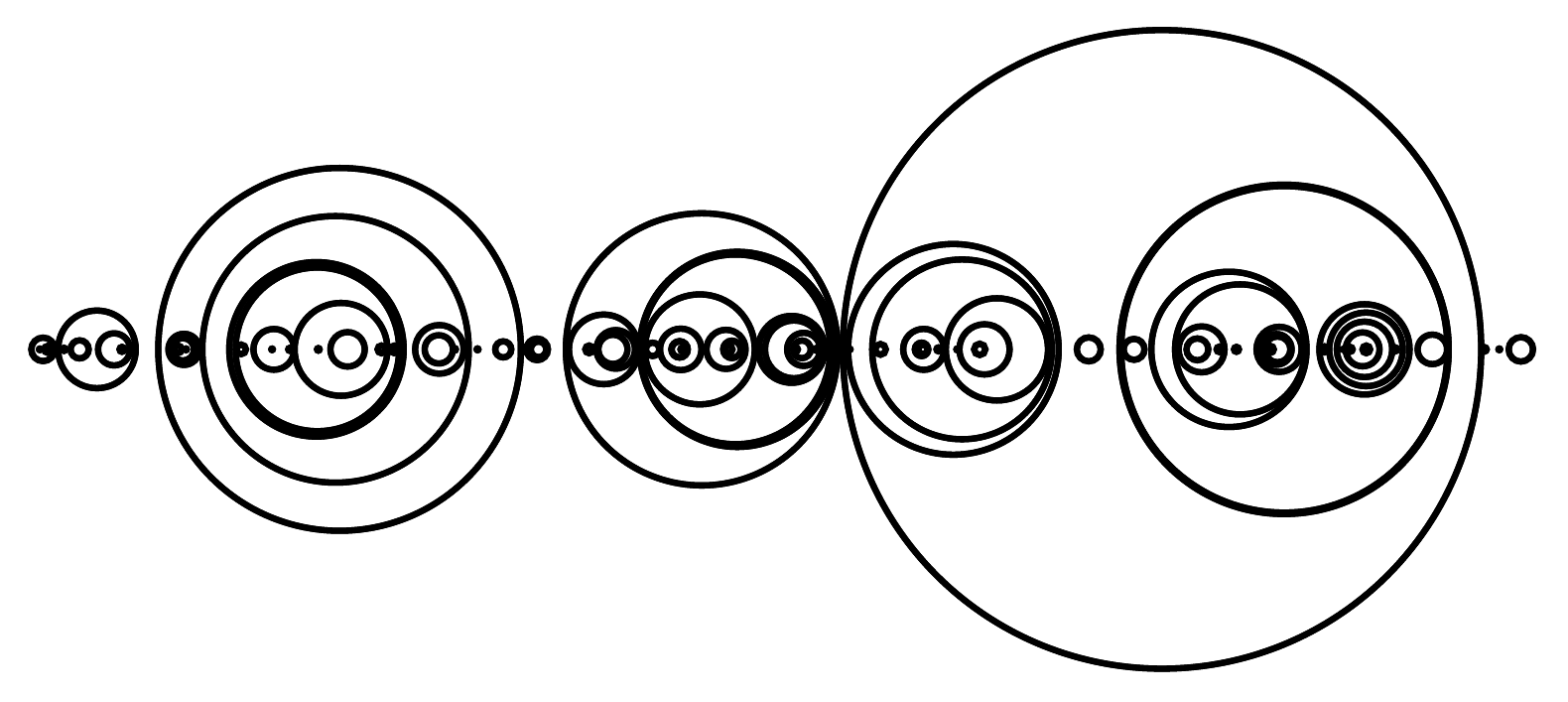}};
          \node at (5.6,-2.85) {Dyck};
        \end{tikzpicture}
        \caption{Four different matchings of the exact same set of $n=100$ points illustrated using McCann circles (\S 2.2): the optimal matching for $c(x,y) = |x-y|^{0.01}$ is very similar to the greedy matching, for  $c(x,y) = |x-y|^{0.99}$ it is close to the Dyck matching.}
        \label{fig:compare}
    \end{figure}
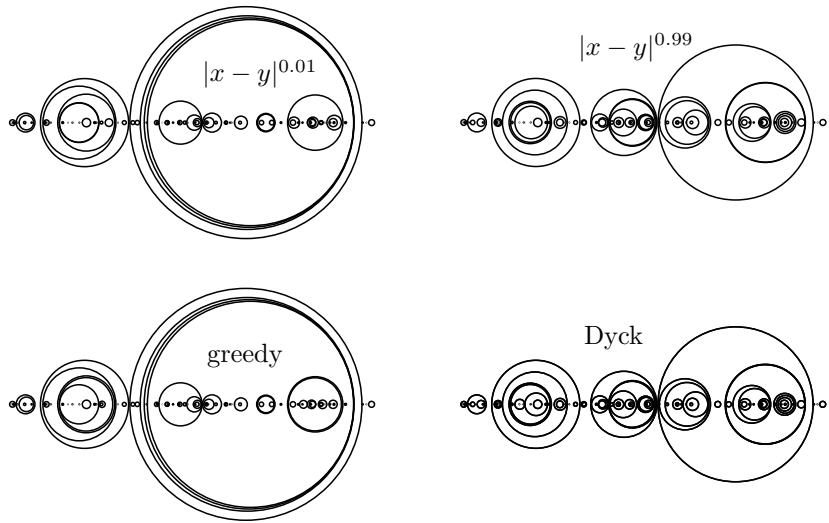
\end{center}

\section{Results}

\subsection{Main Result} It is clear that a greedy matching will initially perform well since it is matching points that are very close to each other. The main concern is that the greedy matching ends up maneuvering itself into a situation where all remaining choices are bad. We will now show that, in a suitable sense, this does not happen. 
The result will be phrased in terms of the Wasserstein distance $W_1$ between the sets $X$ and $Y$ where the sets will be assumed to lie in a metric space
$$ W_1(X,Y) = \inf_{\pi \in S_n} \sum_{i=1}^{n} d(x_i, y_{\pi(i)}).$$
When $0 < p < 1$ it follows from H\"older's inequality with coefficients $1/p$ and $1/(1-p)$ that we can bound the size of the optimal matching by
$$W^p_{p}(X,Y) = \inf_{\pi \in S_n} \sum_{i=1}^{n} d(x_i, y_{\pi(i)})^{p} \leq W_1(X,Y)^{p} \cdot n^{1-p}.$$
Our main result shows that, in any arbitrary metric space, the greedy matching, proceeding at each step blindly and without any foresight into the future, achieves the same rate up to a constant when $0 < p <1/2$. Note that the greedy matching will usually be very different from the one that minimizes $W_1$.
\begin{theorem}[Main Result] For any $0 < p < 1$ there is a constant $c_{p} > 0$ such that for \emph{any} two sets of $n$ points $X,Y$ in \emph{any} metric space, the greedy matching produces matching satisfying, with respect to the cost function $c(x,y) = d(x,y)^{p}$,
$$  \emph{Greedy}_{p}(X,Y) \leq c_{p} \cdot W_1(X,Y)^{p} \cdot
\begin{cases} 
n^{1-p} \qquad &\mbox{if}~0 < p < \frac{1}{2} \\
\sqrt{n} \cdot \log{n} \qquad &\mbox{if}~p=\frac{1}{2} \\
n^{p} \qquad &\mbox{if} ~\frac{1}{2} < p < 1.
\end{cases}
$$
\end{theorem}
The result is optimal up to constants when $0< p < 1/2$. We illustrate this on $[0,1]$ equipped with the Euclidean distance (see Fig. \ref{fig:sharp}). If the blue point are close to $0$ and the red points are close to 1, then the greedy algorithm takes $\mbox{Greedy}_p(X,Y) \sim n$ while $W_1(X,Y) \sim n$. If all the points alternate in an equispaced way, $\mbox{Greedy}_p(X,Y) \sim n^{1-p}$ and $W_1(X,Y) \sim 1$. 
The proof shows that $c_p \sim 1+2p$ for $p$ small.
There is a clear change around $p =1/2$, the greedy matching becomes less effective (see also Fig. \ref{fig:bound}). We also note a result of Bobkov-Ledoux \cite{bob} in one dimension: the optimal transport cost for $d(x,y)^{p}$ and $0 < p < 1$ is, in expectation, smaller than the one induced by the ordered optimal $p = 1$ matching.
\vspace{-10pt}
\begin{center}
    \begin{figure}[h!]
        \centering
     \begin{tikzpicture}
         \node at (0,0) {\includegraphics[width=0.42\textwidth]{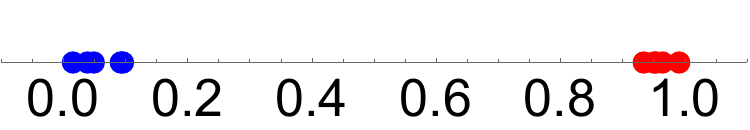}};
            \node at (6,0) {\includegraphics[width=0.42\textwidth]{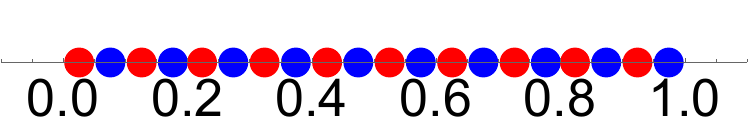}};
     \end{tikzpicture}
        \caption{Examples where the Theorem is sharp ($0 < p < 1/2$).}
        \label{fig:sharp}
    \end{figure}
\end{center}

\subsection{Greedy is non-crossing.} It is known (see McCann \cite{mccann}) that if we match points $\left\{x_1, \dots, x_n\right\}$ to $\left\{y_1, \dots, y_n\right\}$ under a cost function $c(x,y) = h(|x-y|)$ with $h \geq 0$ concave, then the optimal matching satisfies a non-crossing condition which can be described as follows: if the optimal matching sends $x_i$ to $y_{\pi(i)}$, then the $n$ circles $C_i$ that go tangentially though $x_i$ and $y_{\pi(i)}$, the circles
$$ C_i = \left\{z \in \mathbb{R}^2: \left \|z - \frac{x_i + y_{\pi(i)}}{2} \right\| =  \frac{ \left|y_{\pi(i)}-x_i\right|}{2}\right\},$$
do not intersect. The greedy matching has this desirable property for trivial reasons.

\begin{figure}[h!]
\includegraphics[width=0.6\textwidth]{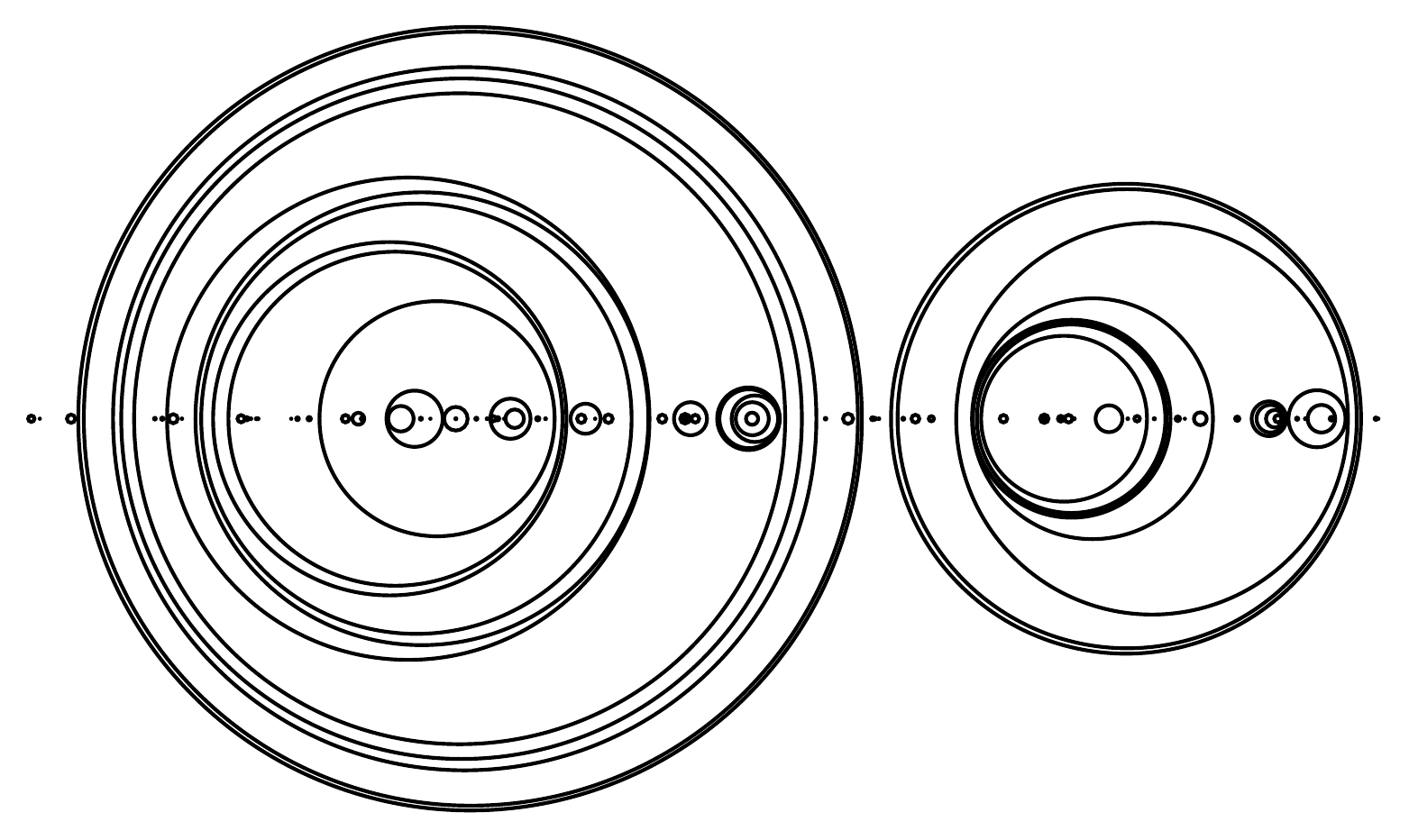}
\caption{The circles $C_i$ for an example of $n=100$ points greedily matched with another $n$ points under $c(x,y) = |x-y|^{0.01}$.}
\end{figure}

\begin{proposition}
    The greedy matching is non-crossing.
\end{proposition}

\begin{proof} The argument is very simple.
    Suppose $i < j$ and circle $C_i$ intersects circle $C_j$. $C_i$ intersects $x_i$ and $y_{\pi(i)}$ while $C_j$ intersects $x_j$ and $x_{\pi(j)}$. Suppose w.l.o.g. $x_i < y_{\pi(i)}$ (otherwise relabel $X$ and $Y$). Since $C_i$ and $C_j$ intersect we have to either have $x_i < x_j < y_{\pi(i)}$ or $x_i < y_{\pi(j)} < y_{\pi(i)}$ but either of these cases leads to a contradiction at stage $i$ of the greedy algorithm.
    \end{proof}

\subsection{Random Points I} One of the most interesting cases is naturally that of matching random points to random points. Here, we can show that the greedy matching leads to nontrivial results.

\begin{corollary}
    Let $X, Y$ be two sets of $n$ uniform i.i.d. random variables on $[0,1]^d$. The greedy matching subject to $c(x,y) = \|x-y\|^{p}$ for $0 < p < 1/2$ satisfies
    $$  \mathbb{E}~ \emph{Greedy}_{p}(X,Y) \leq c_{p} \cdot  
  \begin{cases} 
     n^{1-p} \qquad &\mbox{when}~d=1~\mbox{and}~p<1/4 \\
     n^{1- p/2} \qquad &\mbox{when}~d=1~\mbox{and}~1/4 \leq p <1/2 \\  
    n^{1-p/2}  \qquad  &\mbox{when}~d=2\\
    n^{1 - p/d} \qquad &\mbox{when}~d \geq 3. \end{cases}$$
\end{corollary}
Of these, the first inequality ($d=1$ and $0<p<1/4$) is optimal up to constants. 
It seems likely to assume that, when $d=1$ and $p < 1/2$, we have $ \mathbb{E}~ \mbox{Greedy}_{p}(X,Y) \leq c_{p} \cdot n^{1- p}$. Even stronger results seem to be true: the greedy matching appears to be \textit{remarkably} close to the cost function, even at the pointwise level. The Dyck matching is known to produce an optimal rate.
\begin{theorem}[Caracciolo-D’Achille-Erba-Sportiello \cite{c}] For various different models of $n$ random points on $[0,1]$, we have
$$ \mathbb{E} \left(\emph{Dyck}_n\right) \sim \begin{cases} 
n^{1-p} \qquad &\mbox{if}~0 < p < \frac{1}{2} \\
n^{1/2} \log{n} \qquad &\mbox{if}~p = \frac{1}{2} \\
n^{1/2} \qquad &\mbox{if}~\frac{1}{2} < p \leq 1. \end{cases} 
$$
\end{theorem}
The Dyck matching is an a priori global construction that is naturally connected to the structure of a Brownian bridge, a fact that allows for a variety of tools to be applied. Conversely, the greedy algorithm is a `blind' local algorithm: it is difficult to predict what it will do without running it. As such, proving a result along the lines of Caracciolo-D’Achille-Erba-Sportiello \cite{c} for the greedy matching appears to be difficult and in need of new ideas.

\subsection{Random points II}
More can be said if we assume that the points are chosen uniformly at random with respect to a fixed probability measure $\mu$ since one would then expect a certain limiting behavior to arise. This is indeed the case.

\begin{theorem}[Barthe-Bordenave \cite{barthe}, special case] Let $0<p<d/2$ and $\mu$ be the uniform measure on a bounded set $\Omega \subset \mathbb{R}^d$. If $X,Y$ are i.i.d. copies from $\mu$, then
    $$ \lim_{n \rightarrow \infty}  n^{\frac{p}{d} -1} \cdot W^p_{p}(X,Y) = \beta_{p}(d) \cdot |\Omega|^{\frac{p}{d}}.$$
\end{theorem}

When $d=1$, the result has recently been extended by Goldman-Trevisan \cite{dario} to also allow for randomness with respect to variable absolutely continuous density.

\begin{theorem}[Goldman-Trevisan \cite{dario}, special case] Let $0 < p < 1/2$. There exists $c_{p} > 0$ such that for compactly supported $\mu$ with a.c. density $f(x)dx$
    $$ \lim_{n \rightarrow \infty}  n^{p -1} \cdot W^p_{p}(X,Y) = c_{p} \int_{\mathbb{R}} f(x)^{1-p} dx.$$
\end{theorem}
 The restriction $p < 1/2$ is necessary, the behavior is strictly different at $p = 1/2$ (this is the scale where the fluctuations of the empirical distribution of the points starts to come into play).
 Since the uniform measure is a special case of a measure of the form $f(x)dx$, the implicit (universal) constants are the same as in the result of Barthe-Bordenave in the sense of $ c_{p} = \beta_{p}(1).$
We provide an easy explicit lower bound on these constants. 

\begin{proposition} Let $0 < p < d/2$. Then
    $$ \beta_p(d) \geq    \omega_d^{-\frac{p}{d} } \cdot \Gamma\left( 1 + \frac{p}{d}\right),$$
    where $\omega_d$ is the volume of the unit ball in $\mathbb{R}^d$.
\end{proposition}

The argument is not terribly difficult and somewhat standard for these types of problems, see \cite{steele}. The bound seems to be rather accurate for small values of $p$. As a consequence, we obtain two-sided bounds in the one-dimensional setting when $p$ is close to 0. There is an interesting heuristic: we expect that the optimal matching will send most points distance $\sim c \cdot n^{-1}$ with $c \sim 1$ ranging over some different values but being approximately at order $\sim 1$ and thus the cost should be close to $\sim c^{p} \cdot n^{1-p}$. However, since $p \sim 0+$, one expects $c^{p} \sim 1$ and thus that the optimal transport cost is perhaps given by $(1+\mathcal{O}(p)) \cdot n^{1-p}$. 

\begin{corollary}[Prop. 2 and \cite{c}] For $0 < p< 1/2$,
$$ 2^{-p} \cdot \Gamma(1+p) \leq \beta_1(p) \leq \frac{1}{1-2p}\frac{2^p}{\Gamma(1-p)}$$
\end{corollary}
In particular, we conclude that $\beta_1(p) = 1 + \mathcal{O}(p)$ when $p \sim 0$ is small.

\subsection{Extreme concave matching}
The cost $c(x,y) = |x-y|^{p}$ is particularly natural. There is, in a suitable sense, a canonical limit as $p \rightarrow 0^+$ since
$$ \lim_{p \rightarrow 0^+} \frac{|x-y|^{p} - 1}{p} = \log |x-y|.$$
As suggested by Fig. \ref{fig:2}, the greedy algorithm performs very well in this setting. We prove a basic result suggesting that this is not a coincidence.

\begin{proposition} Let $X,Y$ be two sets of $n$ distinct points in a metric space such that $d(x_i, y_j) \leq 1$ and assume that all pairwise distances are unique. Then there exists $K \in \mathbb{N}$ such that for all $k \geq K$ the solution of the optimal matching problem
    $$ \min_{\pi \in S_n} \sum_{i=1}^{n} \left( \log d(x_i, y_{\pi(i)})\right)^{2k+1} \quad \mbox{    is given by the greedy matching.}$$
\end{proposition}

The argument is quite simple and there is nothing particularly special about the logarithm, similar results could be attained with many other cost functions that are dramatically different across different length scales. The main point of this simple Proposition is to illustrate that the effectiveness of the greedy matching in the setting of very concave cost functions is perhaps not entirely surprising: the dramatic separation of scales puts a heavy reward on having matching with very small distance which, coupled with the separation of scales, then suggests the greedy algorithm as a natural object.

\subsection{Numerics}
The purpose of this section is to consider the behavior of the greedy algorithm and the Dyck matching when matching $n$ i.i.d. random points on $[0,1]$ given the cost function $c(x,y) = |x-y|^{p}$ for $0 \leq p \leq 1/2$. The results suggest that, at least for random points, the greedy matching leads to results that are remarkably close to the ground truth: this effect becomes more pronounced when $p$ becomes smaller (also at least partially suggested by the Proposition).
 
\begin{center}
    \begin{figure}[h!]
        \begin{tikzpicture}
                         \node at (6,0) {\includegraphics[width=0.5\textwidth]{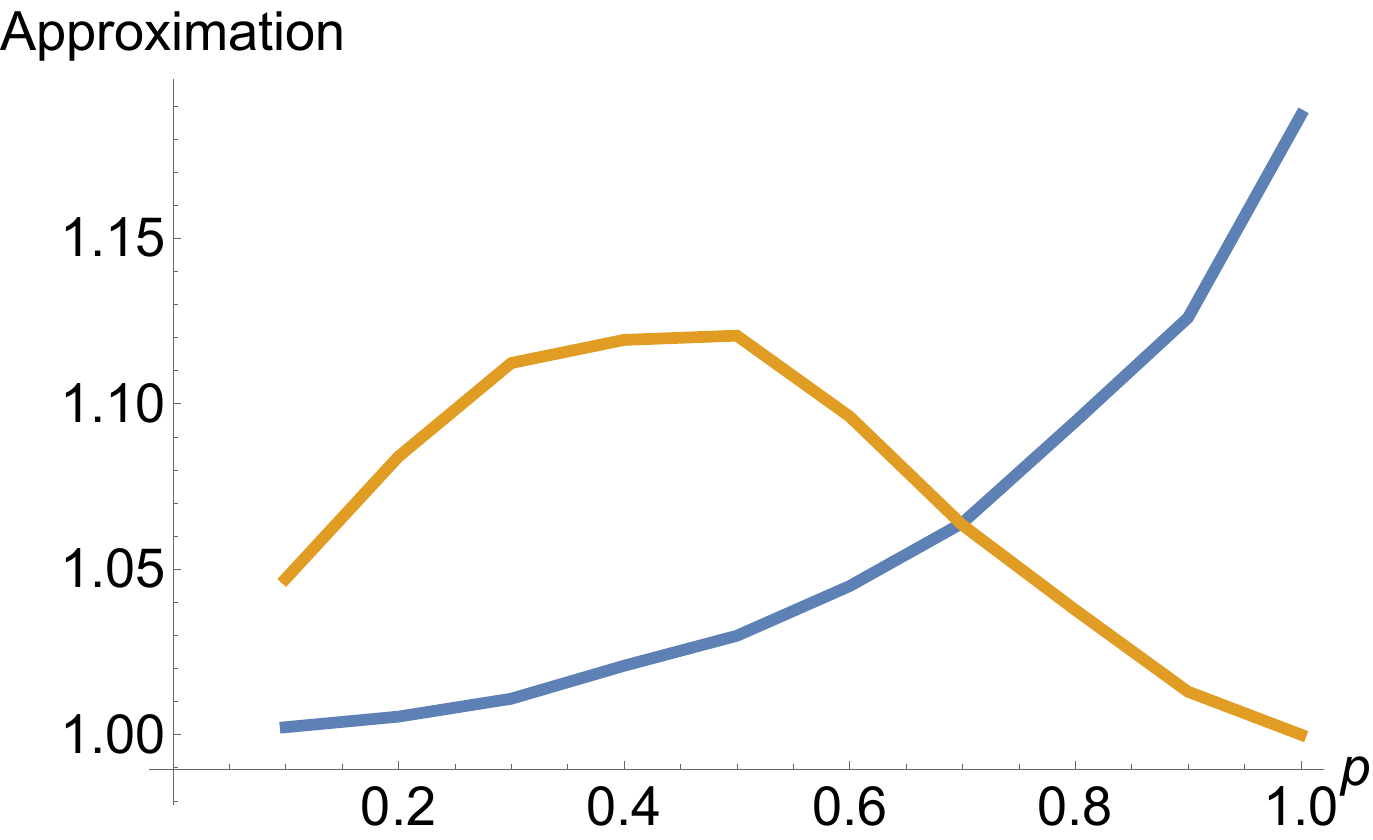}};
                                     \node at (8.2,1.5) {greedy};
             \node at (8.5,-0.9) {Dyck};
        \end{tikzpicture}
        \caption{Average ratio of cost divided by the minimal cost for both the Dyck matching and the greedy matching for $n=250$ iid points subject to the cost function $c(x,y) = |x-y|^{p}$.}
    \end{figure}
\end{center}

The effectiveness of the greedy algorithm is strictly restricted to the region $0 < p < 1/2$: for $p \geq 1/2$, the greedy algorithm starts to scale differently and becomes less effective. This can already be observed for small values of $n$ (and is pointed out by d'Achille \cite[Section 1.4]{achille} for $p \geq 1$). We also observe that, as $p$ becomes smaller, the effectiveness of the greedy matching increases dramatically while, for $p$ close to 1, the effectiveness of the Dyck matching increases dramatically (see Fig. \ref{fig:bound}). The Dyck matching is optimal when $p=1$ (see \cite{c2}).

\begin{center}
    \begin{figure}[h!]
        \begin{tikzpicture}
            \node at (0,0) {\includegraphics[width=0.5\textwidth]{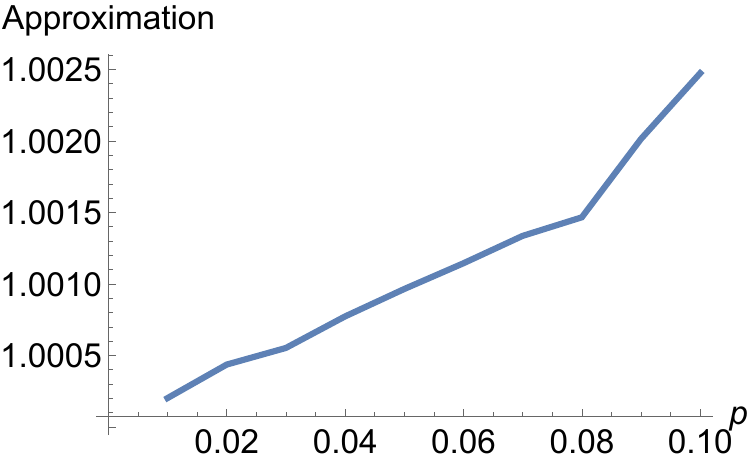}};
             \node at (2,-0.3) {greedy};
                         \node at (6,0) {\includegraphics[width=0.5\textwidth]{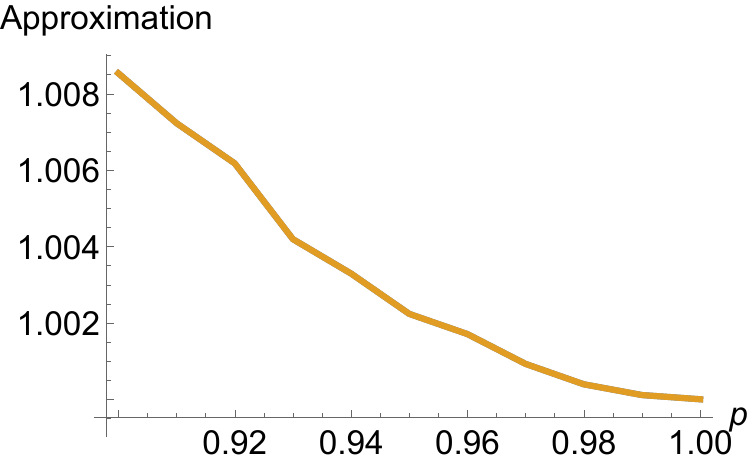}};
             \node at (8.1,-0.8) {Dyck};
        \end{tikzpicture}
        \caption{Ratio of obtained matching cost divided by the minimal cost for $n=100$ random points. Left: the greedy matching when $p \sim 0$. Right: the Dyck matching when $p \sim 1$.}
        \label{fig:bound}
    \end{figure}
\end{center}

 Fig. \ref{fig:bound} is well suited to iterate a main contribution of our paper: when $d=1$ and considering matchings with $c(x,y) = d(x,y)^{p}$ and $0 < p < 1$, there exists a natural dichotomy depending on whether $p$ is close to 0 or whether it is close to 1.
\begin{center}
    \begin{tikzpicture}
        \draw [ultra thick] (0,0) -- (6,0);
        \draw [ultra thick] (0,-0.1) -- (0,0.1);
          \draw [ultra thick] (6,-0.1) -- (6,0.1);
          \node at (0, -0.4) {$0$};
           \node at (3, -0.4) {$p$};
         \node at (6, -0.4) {$1$};
         \node at (0, 0.4) {greedy};
            \node at (6, 0.4) {Dyck};
            \node at (3,0.4) {$c(x,y) = d(x,y)^{p}$};
            \node at (8,0) {($d=1$)};
    \end{tikzpicture}
\end{center}

A natural question is, for example, whether one can identify the precise value $0 < p^* < 1$ where the effectiveness of the two algorithms undergoes a phase transition and the matching problem for $n$ iid random points is better approximated by the Dyck matching or the greedy matching, respectively.

\section{Proof of Theorem 1}

\begin{proof}
We assume the sets of points are $X = \left\{x_1, \dots, x_n\right\}$ and $Y=\left\{y_1, \dots, y_n\right\}$ and we will use $X_k, Y_k$ to denote the set of points after $k-1$ points have been removed following the greedy algorithm.
In particular, $X = X_1$ and $Y= Y_1$. We will also abbreviate the cost at stage $k$ via
$$ c_k = \inf_{x \in X_k, y \in Y_k} d(x, y).$$
Our goal is to estimate
$$ \mbox{the cost of the greedy algorithm} \quad \sum_{k=1}^{n} c_k^{p}.$$
 Recalling that, by definition, $X_k$ and $Y_k$ have $n-k+1$ elements each, the Wasserstein distance $W_1$ can be written as
$$ W_1(X_k, Y_k) = \min_{\pi:X_k \rightarrow Y_k \atop \pi ~{\mbox{\tiny bijective}}} \quad \sum_{x \in X_k}^{} d(x, \pi(x))$$
where the minimum ranges over all bijections. At this point, we remark that the definition of $W_1$ is slightly more comprehensive (the infimum ranges over all ways of splitting and rearranging points): at this point, we employ the celebrated result of Birkhoff \cite{birk} and von Neumann \cite{von} ensuring that in the case of $n$ equal masses being transported to $n$ equal masses, the optimal solution can be realized by a permutation (no mass is `split'), we refer to \cite{bamdad} for a generalization.

Our first observation uses averaging. Let us use $\pi$ to denote the permutation achieving the optimal $W^1$ transport cost. Then the distance achieved by the greedy matching in the $k-$th step can be bounded from above by
\begin{align*}
     c_k = \inf_{x \in X_k, y \in Y_k} d(x, y) \leq \frac{1}{n-k+1}   \sum_{x \in X_k}^{} d(x, \pi(x)) = \frac{W_1(X_k, Y_k)}{n-k + 1}.
\end{align*}
The second ingredient will be to show that $W_1(X_{k+1}, Y_{k+1})$ cannot be much larger than $W_1(X_k, Y_k)$.
For this purpose, let us assume that $X_k$ is given by the points $x_1, x_2, \dots, x_{n-k+1}$ and, similarly, $Y_k$ is given by  $y_1, y_2, \dots ,y_{n-k+1}$. Then, after possibly relabeling the points, we have that
$$ W_1(X_k, Y_k) = \sum_{i=1}^{n-k+1} d(x_i, y_i).$$
Let us now assume that the greedy matching at this point matches $x_i$ and $y_j$, meaning that
$d(x_i, y_j)$ is the smallest pairwise distance in $X_k \times Y_k$. 
Then the greedy matching is going to match these two points up and the remaining sets of points are given by
$$ X_{k+1} = X_k \setminus \left\{x_i \right\} \quad \mbox{and} \quad Y_{k+1} = Y_k \setminus \left\{y_j \right\}.$$
We will provide an upper bound on $W_1(X_{k+1}, Y_{k+1})$ by taking the original matching between $X_k$ and $Y_k$ and then modify it a little to obtain a matching for $X_{k+1}$ with $Y_{k+1}$. This is done by a simple modification: the point $x_j$ is now mapped to $y_i$ (note that $x_i$ has been mapped to $y_j$ and both have been removed from the set). We preserve all other matchings. Then
$$ W_1(X_{k+1}, Y_{k+1}) \leq W_1(X_{k}, Y_{k}) + d(x_j, y_i) - d(x_i, y_i) - d(x_j, y_j).$$
At this point, we invoke the triangle inequality and argue that
\begin{align*}
    d(x_j, y_i) &\leq  d(x_j, y_j) + d(y_j, y_i) \\
     &\leq d(x_j, y_j) + d(y_j, x_i) + d(x_i, y_i)
\end{align*}
Combining the last two inequalities, we realize that we can bound the increase in the $W_1$-distance between the two sets in terms of the cost function of the greedy matching at the $k-$th step by
\begin{align*}
 W_1(X_{k+1}, Y_{k+1}) &\leq W_1(X_{k}, Y_{k}) + d(x_j, y_i) - d(x_i, y_i) - d(x_j, y_j) \\
 &\leq W_1(X_{k}, Y_{k}) + d(x_i, y_j)  \\
 &= W_1(X_{k}, Y_{k}) + c_k.
\end{align*}
Combining these two ingredients, we arrive
\begin{align*}
   W_1(X_{k+1}, Y_{k+1})  &\leq W_1(X_{k}, Y_{k}) + c_k \\
   &\leq W_1(X_{k}, Y_{k}) + \frac{W_1\left(X_k, Y_k\right)}{n-k+1} \\
   &= W_1(X_{k}, Y_{k}) \left( 1 + \frac{1}{n-k+1} \right).
\end{align*}
By induction, we obtain
 $$   W_1(X_{k}, Y_{k}) \leq W_1(X,Y)\cdot \prod_{\ell=1}^{k-1} \left(1 + \frac{1}{n-\ell+1}\right).$$   
Observe that
\begin{align*}
    \prod_{\ell=1}^{k-1} \left(1 + \frac{1}{n-\ell+1}\right) = \prod_{\ell=1}^{k-1} \frac{n-\ell+2}{n-\ell+1} = \frac{n+1}{n-k+2}.
\end{align*}
Therefore
 $$   W_1(X_{k}, Y_{k}) \leq \frac{n+1}{n-k+2} \cdot  W_1(X,Y).$$   
Applying the pigeonhole principle one more time, we see that
$$ c_k \leq  \frac{W_1\left(X_k, Y_k\right)}{n-k+1} \leq \frac{n+1}{(n-k+1)^2} \cdot  W_1(X,Y).$$
Thus
$$ \sum_{k=1}^{n} c_k^{p} \leq  W_1(X,Y)^{p} \cdot (n+1)^{p} \cdot \sum_{k=1}^{n} \frac{1}{(n-k+1)^{2p}}.$$
We have
$$ \sum_{k=1}^{n} \frac{1}{(n-k+1)^{2p}} = \sum_{k=1}^{n} \frac{1}{k^{2p}}.$$
When $0 < p < 1/2$, we have
$$  \sum_{k=1}^{n} \frac{1}{k^{2p}} \leq 1 + \int_1^{n+1} \frac{1}{x^{2p}} dx \leq 1 + \frac{(n+1)^{1-2p}}{1-2p}.$$
Dealing with the remaining cases in the usual fashion, we obtain
$$  \sum_{k=1}^{n} c_k^{p}\leq c_{p} \cdot W_1(X,Y)^{p} \cdot
\begin{cases} 
n^{1-p} \qquad &\mbox{if}~0 < p < \frac{1}{2} \\
\sqrt{n} \cdot \log{n} \qquad &\mbox{if}~p=\frac{1}{2} \\
n^{p} \qquad &\mbox{if} ~\frac{1}{2} < p < 1.
\end{cases}
$$
The argument also shows that, for $p < 1/2$ we have
$$\mbox{Greedy}_{p}(X,Y) \leq \left(\frac{1}{1-2p} + o(1)\right) \cdot n^{1-p} \cdot W_1(X,Y).$$
In particular, for $p$ close to 0, we have that that $1/(1-2p) \sim 1 + 2p$ and the implicit constant
is close to 1.
\end{proof}

\subsection{Proof of Corollary 1}
\begin{proof}
    Corollary 1 follows immediately from the Theorem. The missing ingredient is a good estimate on $W_1(X,Y)$ where $X$ and $Y$ are two sets of $n$ i.i.d. uniformly distributed points in $[0,1]^d$. The case $d=2$ is arguably the most famous, the celebrated result of Ajtai, Komlos, Tusnady \cite{akt} ensures that 
$$ c_1 \sqrt{n \log{n}} \leq \min_{\pi \in S_{n}} \sum_{i = 1}^{n} \| x_i - y_{\pi(i)}\| \leq c_2\sqrt{n \log{n}}.$$
    with high probability. The one-dimensional case is a bit simpler and one has (see, for example, \cite{bobkov}), with high probability,
    $$  \min_{\pi \in S_{n}} \sum_{i = 1}^{n} \| x_i - y_{\pi(i)}\| \leq c  \sqrt{n}.$$

The case $d \geq 3$ where
    $$  \min_{\pi \in S_{n}} \sum_{i = 1}^{n} | x_i - y_{\pi(i)}| \leq c \cdot n^{1-1/d}$$
    was already remarked by Ajtai, Komlos, Tusnady \cite{akt}. A modern treatment of a much more general case is given in \cite{fournier}. This proves that, for $X$ and $Y$ i.i.d. random points and $p<1/2$
    $$  \mathbb{E}~ \mbox{Greedy}_{p}(X,Y) \leq c_{p} \cdot \begin{cases} n^{1- p/2} \qquad &\mbox{when}~d=1 \\  
    n^{1-p/2} \cdot (\log{n})^{p/2} \qquad  &\mbox{when}~d=2 \\
    n^{1 - p/d} \qquad &\mbox{when}~d \geq 3. \end{cases}$$
There is a slight improvement that was pointed out to us by an anonymous referee. For any $0 < \alpha < 1$, we can write
$$ \sum_{i=1}^{n} \|x_i - y_{\pi(i)}\|^{p} = \sum_{i=1}^{n} \left(\|x_i - y_{\pi(i)}\|^{\alpha} \right)^{p/\alpha}.$$
If $d(x,y)$ is a metric, then $d(x,y)^{\alpha}$ is also a metric for every $0 < \alpha < 1$. Moreover, since $x \rightarrow x^{\alpha}$ is monotonically increasing, the greedy algorithm with respect to the metric $d(x,y)^{\alpha}$ picks the exact same matching as the greedy algorithm with respect to $d(x,y)$. We use $W_{1,\alpha}$ to denote the Wasserstein-1 distance with respect to
the metric $d_{\alpha}(x,y) = \|x-y\|^{\alpha}$.
Applying Theorem 1 to this new metric, we get that, assuming $x_i \leftrightarrow y_i$ to be the greedy matching and $0 < p/\alpha <1$ instead of $p$, then
$$ \sum_{i=1}^{n} \|x_i - y_i\|^p \leq c_{p, \alpha} \left( W_{1,\alpha}(X,Y)^{} \right)^{p/\alpha} \cdot
\begin{cases} 
n^{1-p/\alpha} \qquad &\mbox{if}~0 < p/\alpha < \frac{1}{2} \\
\sqrt{n} \cdot \log{n} \qquad &\mbox{if}~p/\alpha=\frac{1}{2} \\
n^{p/\alpha} \qquad &\mbox{if} ~\frac{1}{2} < p < 1,
\end{cases}$$
where
$$W_{1,\alpha}(X,Y)^{} = \inf_{\pi \in S_n} \sum_{i=1}^{n} \|x_i - y_{\pi(i)}\|^{\alpha}.$$
The result of Barthe-Bordenave \cite{barthe} implies that as long as $\alpha < d/2$ and the
$X,Y$ are i.i.d. random variables with respect to the uniform measure, then
$$ \inf_{\pi \in S_n} \sum_{i=1}^{n} \|x_i - y_{\pi(i)}\|^{\alpha} \lesssim n^{1-\alpha/d}.$$
Combining all these results, assuming that $0 < \alpha < 1$ as well as $0 < p < \alpha$ and $\alpha < d/2$, then
$$ \sum_{i=1}^{n} \|x_i - y_i\|^p \leq c_{p,\alpha}   \begin{cases} 
n^{1-p/d} \qquad &\mbox{if}~0 < p/\alpha < \frac{1}{2} \\
n^{1  - p/d} \cdot \log{n} \qquad &\mbox{if}~p/\alpha=\frac{1}{2} \\
n^{p(2/\alpha-d)} \qquad &\mbox{if} ~\frac{1}{2} < p/\alpha < 1.
\end{cases}$$
 The first case improves on the previous estimate when $d=1$, this requires $\alpha < 1/2$ and thus $p < 1/4$. When $d=2$, then for any $p < 1/2$ we can find $p < \alpha < 1$ such that $p/\alpha < 1/2$ which implies the upper bound $n^{1-p/2}$ without logarithmic corrections. Altogether, this gives, for $0 < p < 1/2$
     $$  \mathbb{E}~ \mbox{Greedy}_{p}(X,Y) \leq c_{p} \cdot \begin{cases} 
     n^{1-p} \qquad &\mbox{when}~d=1~\mbox{and}~p<1/4 \\
     n^{1- p/2} \qquad &\mbox{when}~d=1~\mbox{and}~1/4 \leq p <1 \\  
    n^{1-p/2}  \qquad  &\mbox{when}~d=2 \\
    n^{1 - p/d} \qquad &\mbox{when}~d \geq 3. \end{cases}
    $$
    \end{proof}

\subsection{Proof of Proposition 2}
Since the constant is independent of the domain, it will be enough to derive an upper bound in a fixed, arbitrary domain. We choose the unit cube $[0,1]^d$ for convenience, however, we emphasize that there is nothing particularly special about the unit cube.
\begin{lemma}
Given $n$ i.i.d. uniform points $X_1, \ldots, X_n$ in $[0,1]^d$ and an independent uniform point $Y$, we have for all sufficiently small $0<\varepsilon< \varepsilon_0$ that whenever $Y$ is sufficiently far from the boundary of the unit cube $d(Y, \partial [0,1]^d) \geq \varepsilon$, then
$$
\mathbb P(\min_{1 \leq i \leq n} |Y-X_i |\geq \varepsilon)=\left(1- \omega_d\varepsilon^d\right)^{n},
$$
where $\omega_d$ is the volume of the unit ball in $\mathbb{R}^d$. 
\end{lemma}
\begin{proof}
Conditional on $Y$, since the $X_i$'s are all independent, we obtain
\begin{align*}
\mathbb P(\min_{1 \leq i \leq n} |Y-X_i|\geq\varepsilon|Y)=\left(1-|B_Y(\varepsilon)\cap [0,1]^d|\right)^n.
\end{align*}
Since $d(Y, \partial [0,1]^d) \geq \varepsilon$, we have
$$ \left(1-|B_Y(\varepsilon)\cap [0,1]^d|\right)^n = (1 - \omega_d \varepsilon^d)^n.$$
A simple monotonicity argument guarantees that for all $\varepsilon$ the conclusion still holds, up to replacing $=$ with $\geq $.
\end{proof}

\begin{proof}[Proof of Proposition 2] We assume the sets of points $\left\{X_1, \dots, X_n \right\} \subset [0,1]^d$ and
$\left\{Y_1, \dots, Y_n \right\} \subset [0,1]^d$ are both sets of independent uniformly distributed random variables in $[0,1]^d$. 
The main idea is the use of the trivial bound
$$ \mathbb{E} \inf_{\pi \in S_n} \sum_{i=1}^{n} d(x_i, y_{\pi(i)})^{p}  \geq \sum_{i=1}^{n}  \mathbb{E} \inf_{x \in X} |Y_i - x|^p=n\mathbb E\min_{1\leq i\leq n}|X_i-Y|^p$$

We introduce a change of coordinates $\varepsilon = c^p/n^{p/d}$. Then, for all sufficiently small $\varepsilon$ (i.e., for fixed $c, p, d$ and $n\rightarrow +\infty$), we have
\begin{align*}
    \mathbb P\left(\min_{1 \leq j \leq n} |X_i-Y |^p\geq \varepsilon\right) &= \mathbb P\left(\min_{1 \leq j \leq n} |X_i-Y |\geq \frac{c}{n^{1/d}}\right) \\
    &=\left(1 - \frac{\omega_d c^d}{n}\right)^{n} \rightarrow e^{-\omega_d c^d}.
\end{align*}

As we remarked earlier, for \emph{all} $\varepsilon\geq 0$ (i.e., for all $c$) we have
$$
\mathbb{P}\left(\min_{1 \leq j \leq n} |X_i-Y |^p\geq \varepsilon\right)\geq e^{-\omega_dc^p},
$$
so that we obtain, as $n\rightarrow +\infty$, 
\begin{align*}
n \cdot \mathbb E\inf_{1\leq i\leq n}|X_i-Y|^p &=n \int_0^{\infty}  \mathbb P\left(\min_{1 \leq j \leq n} |X_i-Y |^p\geq \varepsilon\right)d\epsilon \\ &=\int_0^{+\infty}pc^{p-1}e^{-\omega_d c^d}dc \\&= n^{1-\frac{p}{d}} \cdot w_d^{-\frac{p}{d}} \cdot \Gamma\left(1+\frac{p}{d}\right).
\end{align*}

We note that the argument slightly improves when $X_i$ is close to the boundary (since then the volume of a neighborhood intersected with $[0,1]^d$ has smaller volume). However, since the number of points distance $\sim n^{-1/d}$ close to the boundary is $\sim n^{(d-1)/d} \ll n$, exploiting this fact would not lead to better asymptotics.
\end{proof}

\subsection{Proof of Proposition 3}
\begin{proof} 
Since $0< d(x_i, y_j) < 1$, all the summands are negative and can use the trivial estimate
\begin{align*}
 n \min_{1 \leq i \leq n} \left( \log d(x_i, y_{\pi(i)})\right)^{2k+1} &\leq \sum_{i=1}^{n} \left( \log d(x_i, y_{\pi(i)})\right)^{2k+1}\\
 &\leq\min_{1 \leq i \leq n}  \left( \log d(x_i, y_{\pi(i)})\right)^{2k+1}.
\end{align*}
Suppose now that 
$$\min_{1 \leq i \leq n} d(x_i, y_{\pi(i)}) > \min_{1 \leq i,j \leq n} d(x_i, y_{j}).$$
Then, for all $k \geq K_1$ sufficiently large, 
$$n \min_{1 \leq i \leq n} \left( \log d(x_i, y_{\pi(i)})\right)^{2k+1} >\min_{1 \leq i,j \leq n}  \left( \log d(x_i, y_{j})\right)^{2k+1}$$
since one grows exponentially larger than the other. This shows that any optimal matching has to at least coincide with the greedy matching in the first step. We emove the closest pair of points $(x_i, y_j)$ and repeat the procedure on the remaining set of points. We see that for all $k \geq K_2$, the next step has to coincide with that of the greedy matching. Repeating the procedure, we see that for all $K \geq \max(K_1, \dots, K_{n-1})$ the minimum can only be given by the greedy matching. 
\end{proof}

\textbf{Acknowledgment.}
The authors gratefully acknowledge support from the Kantorovich Initiative. A.O. is supported by an AMS-Simons Travel Grant. S.S. was supported by the NSF (DMS-2123224). The authors are indebted to Aleh Tsyvinski for drawing their attention to the problem and to an anonymous referee for suggesting improved rates for Corollary 1.

\end{document}